\documentclass[preprint,3p,authoryear,a4paper]{elsarticle}
\usepackage{graphicx}
\usepackage{parskip}
\usepackage{xcolor}
\usepackage{dsfont}
\usepackage{amsmath}
\usepackage{amssymb}
\usepackage{amsfonts}
\usepackage{amsbsy}
\usepackage{amsthm}
\usepackage{varioref}
\usepackage{array}
\usepackage{booktabs} 
\usepackage{verbatim} 
\usepackage[english]{babel}
\usepackage{float}
\usepackage{subfigure}
\usepackage{epsfig}
\usepackage{multirow}
\usepackage{pdflscape}
\usepackage{mathrsfs}
\usepackage{lscape}
\usepackage{ulem}

\newcolumntype{C}[1]{>{\centering\let\newline\\\arraybackslash\hspace{0pt}}m{#1}}
\renewcommand\appendix{\par
\setcounter{section}{0}%
\setcounter{subsection}{0}%
\setcounter{table}{0}
\setcounter{table}{0}
\setcounter{figure}{0}
\gdef\thetable{\Alph{table}}
\gdef\thefigure{\Alph{figure}}
\gdef\thesection{\Alph{section}}
\setcounter{section}{0}}

\newtheorem{theorem}{Theorem}[section]

\newtheorem{proposition}[theorem]{Proposition}
\newtheorem{corollary}[theorem]{Corollary}

\newenvironment{arcitem}{
\begin{list}{---}{
	\topsep=1pt 
	\itemsep=1pt 
	\parsep=0pt 
	\leftmargin=19pt 
}}
{\end{list}}

\newcounter{arclist}

\newcounter{arcenum}

\newcommand{\td}{\text{d}}

\newcommand{\EE}{\mathbb{E}}
\newcommand{\JJ}{\mathcal{J}}

\newcommand{\PQ}{\mathbb{P} \otimes \mathbb{Q}}
\newcommand{\RR}{\mathbb{R}}

\newcommand{\bfi}{\textbf{i}}

\begin{document}

\begin{frontmatter}

\title{
Stable Dividends under Linear-Quadratic Optimization
}

\author[UM]{Benjamin Avanzi}
\ead{b.avanzi@unimelb.edu.au}

\author[UM,Cop]{Debbie Kusch Falden \corref{cor}}
\ead{dkf@math.ku.dk}

\author[Cop]{Mogens Steffensen}
\ead{mogens@math.ku.dk}


\cortext[cor]{Corresponding author. }

\address[UM]{Centre for Actuarial Studies, Department of Economics, Faculty of Business and Economics \\ University of Melbourne, Melbourne VIC 3010, Australia}

\address[Cop]{Department of Mathematical Sciences, University of Copenhagen \\ DK-2100 Copenhagen, Denmark}

\begin{abstract}

The optimization criterion for dividends from a risky business is most often formalized in terms of the expected present value of future dividends. That criterion disregards a potential, explicit demand for stability of dividends. In particular, within actuarial risk theory, maximization of future dividends have been intensively studied as the so-called de Finetti problem. However, there the optimal strategies typically become so-called barrier strategies. These are far from stable and suboptimal affine dividend strategies have therefore received attention recently. In contrast, in the class of linear-quadratic problems a demand for stability if explicitly stressed. These have most often been studied in diffusion models different from the actuarial risk models. We bridge the gap between these patterns of thinking by deriving optimal affine dividend strategies under a linear-quadratic criterion for a general Lévy process. We characterize the value function by the Hamilton-Jacobi-Bellman equation, solve it, and compare the objective and the optimal controls to the classical objective of maximizing expected present value of future dividends. Thereby we provide a framework within which stability of dividends from a risky business, as e.g. in classical risk theory, is explicitly demanded and explicitly obtained.

\end{abstract}


\end{frontmatter}

\newtheorem{remark}{Remark}[section]
\numberwithin{equation}{section}

\section{Introduction} \label{S_intro}

We formalize a dividend optimization problem within a risk theoretical framework where a demand for stability is explicitly expressed via a linear-quadratic objective function. The solution is affine dividends and, thus, we contribute with an understanding of the class of objectives where the recently studied affine dividend processes are actually optimal. We realize that the diffusion approximation of a general Lévy process is adequate under this criterion. Finally, we illustrate and discuss the performance of our optimal dividends under the standard \citet{deF57} criterion where stability is neither expressed nor obtained.

\citet{Buh70} discussed, in the context of classical actuarial risk theory, the problem of offering stability to investors providing capital to a risky business. The insurance business is a core example where risk transfer is the rational for making business the discussion goes beyond that business. Historically, in risk theory the criteria have been to minimize the probability of ruin as in \citet{AsAl10} or to maximize the expected present value of dividends until ruin as in \citet{AlTh09,Ava09}. The general endeavour of the stability problem is still a contemporary one and can be seen through the lens of modern Enterprise Risk Management, described in \citet{Tay132}. Essentially, Enterprise Risk Management is a process by which certain decisions are made (the controls) to achieve certain outcomes (the objective), possibly under certain constraints that are either external to the decision maker (e.g., coming from a regulatory environment) or internal to the decision maker (and in that case not necessarily truly distinguishable from other aspects of the objective). Such a management procedure can be advantageously presented by stylised modelling as the one developed in the prolific risk theoretical literature; see also \citet[Section 1.4, on the value of simple but tractable models]{Cai00} and \citet[on the value of ruin theory for risk managers, in particular for capital modelling]{GeLo12}

Stability criteria are most often linked to the surplus of a company, for instance, and, originally, related to minimizing the probability of ruin corresponding to the probability of a negative surplus. Historically, risk theoretical surplus models focused on insurance type dynamics where the risk is downside risk with deterministic income and stochastic losses. The classical formulation is the Cram\'er-Lundberg model, which is a compound Poisson surplus model formalized by \citet{Lun09,Cra30}. Recently, more general risky business types have been considered, for instance, where the stochastic nature is mostly on the upside as gains \citep{AvGeSh07,BaEg08}. The most general formulations are in terms of spectrally negative or positive L\'evy processes \citep{Loe08,BaKyYa13}. The generalization from the classical compound Poisson model to a more general surplus process opens up for applying the patterns of thinking far beyond the insurance business.

The objective of minimizing the probability of ruin over infinite time implies that the surplus increases without a limit. In order to resolve this issue \cite{deF57} allowed for a surplus leakage to the shareholders of the company, referred to as dividends, and formed stability criteria based on the future dividend payouts. The classical objective is to maximize the expected present value of future dividends until the company is ruined. 
This coincides with the Dividend Discount Model by \cite{Wil38}, also known as the Gordon Model in finance \citep{Gor62}. 
\citet{Loe09} and \citet{YiWe13} studied optimal dividend problems within the classical objective for spectrally negative L\'evy processes. Realistic features of both the controlled process, the control, and the objective are, still, being added to the dividend optimization criterion. \citet{AvTuWo16c} developed a list of realistic features one might want to include, in particular based on the corporate finance literature. One important and well-known aspect is that companies and investors like stable dividends \citep[see, e.g.,][]{Lin56,FaBa68,AvTuWo16c}. Unfortunately, the optimal strategies for the dividend optimization problem turn out to be relatively irregular, namely the so-called barrier strategy, hardly acceptable in practice. This was first pointed out by \citet{Ger74}, but received little attention until recently. In an attempt to address this issue, \citet{AvWo12} introduced a linear dividend strategy, in a diffusion framework, leading to mean reversion of the surplus process and much improved stability. This was generalised to affine dividend strategies by \citet{AlCa17}, in a Cram\'er-Lundberg framework, who also derived a closed-form Laplace transform of the time to ruin. Importantly, both \citet{AvWo12} and \citet{AlCa17} illustrate that affine dividend strategies perform closely to the optimal barrier strategies, but they make the surplus process much more stable. This latter point is more rigorously explored in \citet{AlCa17} by their theoretical analysis of ruin. 

In both \citet{AvWo12} and \citet{AlCa17} the linear and dividend strategies, respectively, are introduced ad hoc and not discussed as solutions to any optimal control problem. Optimal parameters, within the ad hoc specified strategy classes, that maximise the expected present value of dividends are obtained. However, affine strategies have been found optimal in different but related contexts by \citet{Cai00} and \citet{Ste06}, where objectives of linear-quadratic form (LQ optimization) are studied in the context of life insurance and pensions. Affine dividend strategies are arguably much more realistic than the usually optimal ones such as barrier strategy. In this paper, we establish a connection risk theory and the class of models typically considered there, and linear-quadratic optimization. In order to show that the linear-quadratic objective entails affine optimal strategies also in rather general risk models, we characterize the value function by the so called Hamilton-Jacobi-Bellman equation. The quadratic form of the value function in terms of the surplus leads to affinity of the optimal dividends payouts. These strategies are obviously suboptimal in relation to dividend optimization but we are able to calculate explicitly the value of the optimal affine dividends coming out of our problem. This allows for a comparison of the performance of our affine dividends in the dividens optimization problem, however, with particular attention to the fact that the classical dividend optimization problem is stopped upon ruin whereas ours is not.

The paper is organised as follows. Section \ref{Sec: setup} introduces the surplus process and the LQ objective that we propose in this paper is analysed and motivated. We derive the Hamilton-Jacobi-Bellman equation and an appropriate verification lemma in Section \ref{Sec: HJB}, along with an expression that characterizes the value function. The LQ objective is compared to the classical objective in Section \ref{Sec: comparison}, where we also study choices of benchmarks in the LQ problem and the resulting optimal dividend strategy. Numerical illustrations are provided in Section \ref{Sec: numerical}.

\section{The optimization problem} \label{Sec: setup}

\subsection{The surplus model}
We model the surplus of a company at time $t$ after distribution of dividends by the dynamics
\begin{equation}\label{E_surplus}
dX(t) =  c(t)dt + d S(t) - dD(t),
\end{equation}
where $c(t)$ is deterministic and represents the predictable modification component of the surplus due to income and expenses, $S(t)$ is stochastic and represents the aggregate random variations of the surplus due to, for instance, losses with $S(0)=0$, and $D(t)$ is the aggregated net dividends with $D(0)=0$. 

If $c(t)$ is a positive constant and $S(t)$ is a compound Poisson process with negative jumps only, then \eqref{E_surplus} has the dynamics of a Cram\'er-Lundberg process. Conversely, if $c(t)$ is a negative constant and $S(t)$ is a compound Poisson process with positive jumps only, then \eqref{E_surplus} has the dynamics of a so-called dual model \citep{MaRu04}. 

We assume $S(t)$ is the following process 
\begin{equation}
S(t) = \sum_{i=1}^{N(t)} Y_i + \varsigma W(t), \quad N(0)=W(0)=0,
\end{equation}
where $(Y_i)_{i\in\mathbb{N}}$ is i.i.d. and $Y_i$ can follow any distribution on $\mathbb{R}$, with finite first two moments,
\begin{equation}
E[Y_i^j]=p_j, \quad j=1,2,
\end{equation}
where $W(t)$ is a Brownian motion, and where $N(t)$ is an inhomogeneous Poisson process with intensity $\lambda(t)$, $t\ge 0$. Note that $S(t)$ is a L\'evy process for $\lambda(t) = \lambda$, where $N(t)$ is a homogeneous Poisson process. Such two-sided formulations are rare in the actuarial literature, but they exist; see \citet[for references with negative and positive $c(t)$, respectively]{Che11b,LaSeSe11} or \citet{ChLiWi18}.

The dividend process, $D(t)$, is not strictly increasing. Hence, we allow negative dividends, spoken of as capital injections. Furthermore, dividends and capital injections can be paid continuously or as lump sums upon jumps in $S(t)$, such that the dynamics of $D$ is given by
\begin{align*}
    \td D(t) = l(t,X(t)) \td t + i(t,X(t-)) \td N(t).
\end{align*}

\subsection{The Linear-Quadratic (LQ) objective}
\label{SubSec: LQ objective}
We consider a finite time frame $T\geq0$ and would like to consider a general objective of the form 
\begin{eqnarray*}
\min  E_{t,x}\Big[ && \text{discounted penalties for continuous dividends away from a benchmark} \\
&+&\text{discounted penalties for lump sum (discrete) dividends} \\
&+&\text{discounted penalties for the wealth process away from a benchmark} \\ 
&+&\text{subject to a constraint on terminal wealth }X(T) \quad\quad \Big],
\end{eqnarray*}
where the subscript of the expectation refers to the expectation conditional of $X(t)=x$. This is operationalised into the following value function,
\begin{eqnarray}\label{Eq: value function}
V(t,x) = \min_{l,i} E_{t,x} \bigg[ && \tfrac{1}{2} \int_t^T e^{-\delta (s-t)} \Big(l(s,X(s)) - l_0(s)-l_1(s)X(s)\Big)^2 ds  \nonumber \\
&+& \tfrac{1}{2} \int_t^T e^{-\delta (s-t)} \gamma^i(s) i(s,X(s-))^2 d N(s) \nonumber  \\
&+& \tfrac{1}{2} \int_t^T e^{-\delta (s-t)} \Big( X(s) - x_0(s)\Big)^2 d\Gamma(s)   \nonumber      \\ 
&+&\kappa e^{-\delta (T-t)} \left(X(T)-x_T\right)^\tau \quad \bigg],
\end{eqnarray}
 for $t\leq T$, where $\delta$ is a financial impatience factor. To get a better understanding of the objective behind this value function, we explain \eqref{Eq: value function} line by line. 
 \begin{arcitem}
\item The first line compares the continuous payout of dividends with an affine benchmark. Dividends are generally not paid continuously, but we use a continuous model that provides a tractable stylised formulation of a discrete real life situation; for comments about this see \citet{Cai00}. The benchmark, $l_0(s)+l_1(s)X(s)$, consist of two functions, a fixed target, $l_0(s)$, and a target which is proportional to the surplus level, $l_1(s)$. 

\item The second line accounts for lump sum payments. The lump sum payments are interpreted as extra dividends or capital injections paid on top of the regular dividends. Therefore, the only admissible lump sum payments are upon jumps in the surplus process, where an abrupt change of surplus level due to a jump may require a discrete adjustment of the surplus. 
The benchmark is zero, since we prefer not to have lump sum dividends, and we introduce a weight function $\gamma^i(s)\ge0$, $s\ge 0$ to adjust this preference. The undesirable signals of lump sum dividends are discussed in \cite{AvTuWo16} and \cite{AvPeWoYa16}. The squared function means we equally dislike lump sum dividends and capital injections, and that we prefer a series of small dividend payouts to one single large one.

\item The last two lines consider the surplus process. The third line compares the surplus with a surplus benchmark. The benchmark, $x_0(s)$, could be a result of regulatory constraints or correspond to the explicit target capitalisation of the company. Companies often set and publish such targets; see, e.g. \citet[for insurance companies]{AI16}. It is reasonable to assumes the function $x_0$ is non-negative, to not aim for the company to ruin. In order to balance this objective with the first two lines, the third line contains a mixed aggregate weight function $\Gamma(t) = \int_0^t \gamma(s) ds$, $0\le t< T$. It is written using the Riemann-Stietjes notation to allow for a final mass at termination $\Delta \Gamma(T)\ge 0$.  

\item The last line serves to control the terminal value of the surplus to a benchmark, $x_T$. The parameter $\kappa$ is a Lagrangian multiplier, and the parameter $\tau$ allows for three levels of constraints on the terminal value $X(T)$. The case $\tau=0$ corresponds to absence of constraint. If $\tau=1$, the expected value of the terminal value is $x_T$, where $\kappa$ is solved to satisfy this constraint. For $\tau=2$ the constraint is stronger and the process is forced to reach $x_T$ at time $T$ by letting $\kappa$ go to infinity. See \citet{Ste06} and \citet{Stphd01} for details.
\end{arcitem}

The final weight function mass $\Delta \Gamma(T)$ is not redundant with the last row for $\tau=2$, since the third row expresses a preference and the fourth row expresses a constraint. Therefore, they are operationalised differently, where the weight at $\Delta \Gamma(T)$ remains a finite constant, while $\kappa$ is meant to diverge in the constraint, such that $X(T)$ is exactly $x_T$. Furthermore, we can have the constraint with $\tau=1$, and use the weight $\Delta \Gamma(T)$ to express a strong preference for the terminal value of the surplus without the binding constraint of $\tau=2$. 

Except for the last line, all distances from the dividends and the surplus to the benchmarks, respectively, are penalised by a quadratic loss function. Objectives on this form are well known in the quantitative finance literature such as \cite{Wo68} and \cite{Bj09}, and optimization in this context is commonly referred to as ``linear-quadratic (LQ) optimization''. LQ optimization is most commonly formalized with an underlying diffusion process without jumps and mainly considered in the context of pensions funds within actuarial risk theory \citep{Cai00,Ste06,AvLaSt22}. It is also well known that LQ optimization results in affine optimal strategies, which induces the desire to understand the objective and the resulting dividend strategy in a broader actuarial context. The objective and optimal controls are relevant to compare to the classical objective of maximizing expected present value of future dividends. In order to show that the LQ objective leads to affine optimal strategies, we characterize the value function by differential equations, and express the value function as a quadratic function in the surplus.

\section{HJB equation and verification lemma}
\label{Sec: HJB}
\subsection{HJB equation and verification lemma for the LQ objective}

Under the assumption that the optimal control strategies exist, the value function satisfies a system of differential equations, referred to as the Hamilton-Jacobi-Bellman equation (HJB equation). The optimal control strategies are the functions $t\mapsto l(t,X(t))$ and $t\mapsto i(t,X(t))$ that minimize the value function and are predictable with respect to the filtration generated by the surplus process. The subscript of a function refers to the partial derivative with respect to that subscript i.e. $V_t(t,x) = \frac{\partial}{\partial t} V(t,x)$. 

\begin{proposition}
\label{prop: HJB equation}
Assume the value function is twice continuously differentiable, $V \in C^{1,2}$ and the optimal control strategies exist. Then the value function satisfies the HJB equation
\begin{align}
\label{eq: HJB equation}
0 =&\  V_t(t,x) - \delta V(t,x) + \inf_{l,i} \Bigg\{ \  \frac{1}{2} \Big( l(t,x) - l_0(t) - l_1(t)x\Big)^2 \nonumber \\
& \hspace{4 cm} + \frac{1}{2} \gamma^i(t) i (t,x)^2\lambda(t) \nonumber \\
    & \hspace{4 cm} + \frac{1}{2} \gamma(t) \Big( x- x_0(t)\Big)^2  \nonumber \\
    & \hspace{4 cm} + V_x(t,x) \Big( c(t) - l(t,x)\Big)\nonumber \\
     & \hspace{4 cm} + \frac{1}{2} V_{xx}(t,x) \varsigma^2 \nonumber \\
     & \hspace{4 cm} +  \lambda(t) \EE \Big[V(t, x+ Y_{1}-i(t,x)) - V(t, x) \Big]
    \Bigg\},  \end{align}
    with boundary condition
    \begin{align}
    \label{eq: HJB equation boundary condition}
        V(T,x) =&\ \kappa (x - x_T)^\tau   + \Delta \Gamma(T)(x-x_0(T))^2.
    \end{align}
For each $(t,x)\in[0,T]\times \RR$ the infimum is attained by the optimal control strategies, and $Y_1$ is one of the stochastic variables in the jump process representing a jump size.
\end{proposition}

\begin{proof}
See Appendix \ref{App: Proof of HJB equation}
\end{proof}

The HJB equation characterizes the value function if the conditions in Proposition \ref{prop: HJB equation} are satisfied, but the converse is also true. It is a sufficient condition such that if a function satisfies the HJB equation, it is the value function.

\begin{proposition}
\label{prop: Verification}
Assume a function $H$ satisfies the HJB equation
\begin{align*}
0 =&\  H_t(t,x) - \delta H(t,x) + \inf_{l,i} \Bigg\{ \  \frac{1}{2} \Big( l(t,x) - l_0(t) - l_1(t)x\Big)^2 \nonumber \\
    & \hspace{ 4.2 cm} + \frac{1}{2} \gamma^i(t) i(t,x)^2 \lambda(t) \nonumber \\
    & \hspace{ 4.2 cm} + \frac{1}{2} \gamma(t) \Big( x- x_0(t)\Big)^2  \nonumber \\
    & \hspace{ 4.2 cm} + H_x(t,x) \Big( c(t) - l(t,x)\Big)\nonumber \\
     & \hspace{ 4.2 cm} + \frac{1}{2} H_{xx}(t,x) \varsigma^2 \nonumber \\
     & \hspace{ 4.2 cm} +  \lambda(t) \EE \Big[H(t, x+ Y_{1}-i(t,x)) - H(t, x) \Big]
    \Bigg\},    
    \end{align*}
    with boundary condition
    \begin{align*}
         H(T,x) =&\  \kappa (x - x_T)^\tau   + \Delta \Gamma(T)(x-x_0(T))^2, 
    \end{align*}
and $H_x(t,X)\varsigma \in \mathcal{L}^2$.
Furthermore, assume that the infimum is attained by admissible control strategies $\tilde{l}$ and $\tilde{i}$ for each fixed $(t,x)$. Then the optimal value function to the control problem, $V$ from Equation \eqref{Eq: value function}, is
\begin{align*}
    V(t,x)=H(t,x),
\end{align*}
and the optimal control strategies are $l^*=\tilde{l}$ and $i^*=\tilde{i}$.
\end{proposition}

\begin{proof}
See Appendix \ref{App: Proof of verification}
\end{proof}

The HJB equation is given as the infimum over admissible controls of a partial differential equation (PDE). By the quadratic structure of the HJB equation, the infimum is not obtained in the limits of the admissible controls going to infinity or minus infinity. Therefore, in order to find expressions for the optimal control strategies, we consider the critical point, where the partial derivatives of the expression in the curly brackets with respect to $l$ and $i$ both equal 0 
    \begin{align}
        &l^*(t) = l_0(t) + l_1(t)x + V_x(t,x), \label{Eq: optimal l}\\
        &i^* \gamma^i(t) - \EE\Big[V_x(t,x+Y_{1}-i^*)\Big]  = 0, \label{Eq: optimal i}
    \end{align}
for sufficiently regular $V$. Hence, the optimal continuous dividend payment is equal to the benchmark plus the derivative of the value function with respect to the surplus, and the optimal dividend payment upon jumps, is related to the expectation of the derivative of the value function with respect to the surplus after a jump. The optimal controls minimize the value function if the second derivative of the expression in the curly brackets with respect to $l$ and $i$ is positive. This is true for $l$, where the second derivative equals 1, but we need to make sure $\gamma^i(t) + \EE\Big[V_{xx}(t,x+Y_{1}-i)\Big] >0$. Based on these expressions, it is not clear that the dividend strategies in the LQ optimization problem have an affine structure. However, we are able next to express the value function as a quadratic function, which shows the affine optimal dividend strategies. 

\subsection{Quadratic value function and affine optimal dividend strategy}
\label{Subsec: alt representation}

The sufficient condition of satisfying the HJB equation serves as a verification lemma, such that we are able to characterize the value function by a function that satisfies the HJB equation. We guess a solution to the HJB equation based on separation of $x$ inspired by \citet{Cai00} and \citet{Ste06}.
 \begin{proposition}
 \label{Prop: alt representation}
 The value function in Equation \eqref{Eq: value function} is given by
 \begin{align}
     V(t,x) = q(t)x^2 + p(t)x + r(t)
     \label{Eq: alt value function}
 \end{align}
 where the functions $q$, $p$ and $r$ satisfy the system of ODEs stated in Appendix \ref{App: q p r} along with the stated terminal conditions. 
 \end{proposition}
 
 \begin{proof}
 Assume the deterministic functions  $q, p$ and $r$ satisfy the system of ODEs in Appendix \ref{App: q p r}. The function $ V(t,x) = q(t)x^2 + p(t)x + r(t)$ then satisfies the HJB equation in Proposition \ref{prop: Verification} and $V_x(t,X)\varsigma \in \mathcal{L}^2$, which implies the result by Proposition \ref{prop: Verification}. The terminal conditions are obtained by considering
  \begin{align*}
    V(T,x) = \kappa (x - x_T)^\tau   + \Delta \Gamma(T)(x-x_0(T))^2
\end{align*}
for all values of $x\in \RR$.  
 \end{proof}

 The expression for the value function in Equation \eqref{Eq: alt value function} implies that the optimal controls from Equations \eqref{Eq: optimal l} and \eqref{Eq: optimal i} are affine in the surplus $x$:
\begin{align}
l^*(t) &= l_0(t) + l_1(t)x + 2q(t)x+p(t), \label{Eq: optimal l affine}\\
     i^* &= \frac{ 2q(t)x+2q(t)p_1+p(t)}{\gamma^i(t)+2q(t)}, \label{Eq: optimal i affine}
 \end{align}
 where we need the second order condition $\gamma^i(t)+2q(t)>0$ for the optimal controls to minimize the value function. Therefore, the objective described in Section \ref{SubSec: LQ objective} results in the desirable affine dividend strategies as the optimal strategies.

The HJB equation from Proposition \ref{prop: HJB equation} characterizes the value function by a PDE and expresses the dividends strategy in terms of the derivative of the value function with respect to the surplus, which is computationally demanding to calculate if even possible. The quadratic representation in Proposition \ref{Prop: alt representation} reduces the dimension of the PDE from Proposition \ref{prop: HJB equation} to a system of ordinary differential equations (ODE) and expresses the optimal dividend strategy in terms of the solutions to the ODEs and as affine functions in the surplus. The ODEs in Appendix \ref{App: q p r} fit into the class of Riccati equations. It is not certain that Riccati equations have solutions, but if a solution exists it is relatively easy to solve the system of ODEs numerically, for instance, by Runge-Kutta methods. Hence, we are able to compute the value function for any given $(t,x)\in[0,T]\times \RR$ by solving the system of ODEs, but we are in general not interested in the explicit value of the value function. We are interested in understanding the objective, and the resulting optimal dividend strategy, which can be expressed in terms of $q(t)$, $p(t)$, and $r(t)$.

\subsection{Coincidence with the HJB equation based on a diffusion surplus model}
\label{SubSec: diffusion surplus model}

The HJB equation in Proposition \ref{prop: HJB equation} is similar to the HJB equation obtained in \citet{Cai00} and \citet{Ste06}, with additional terms emerging from changes in the underlying surplus model caused by jumps and dividend payments upon the jumps. With the value function expressed as a quadratic function in Proposition \ref{Prop: alt representation}, we can express the last line of the HJB equation in terms of the derivative and second derivative of the value function with respect to the surplus 
\begin{align}
\label{Eq: expectation of jump}
    &  \lambda(t)  \EE \Big[V(t, x+ Y_{1}-i(t,x)) - V(t, x) \Big] \nonumber\\ 
    & = \EE \Big[\Big(2q(t)x+p(t)\Big)\lambda(t)(Y_1-i(t,x)) +q(t)\lambda(t)(Y-i(t,x))^2 \Big] \nonumber\\
    &= V_x(t,x)\lambda(t)(p_1-i(t,x)) +\frac{1}{2}V_{xx}(t,x)\lambda(t)(p_2 + i(t,x)^2 - 2 p_1 i(t,x)).
\end{align}

Since the objective is based on expectation and from Equation \eqref{Eq: expectation of jump}, we can model the surplus as a diffusion process without jumps and obtain the same HJB equation and optimal control strategies. 

\begin{corollary}
\label{coro: conside with surplus without jumps}
Let the surplus, $\hat{X}$, have dynamics
\begin{align*}
    \td \hat{X}(t) = \Big( c(t) -l(t,\hat{X}(t)) + \lambda(t) \big(p_1 - i(t,\hat{X}(t))\big)\Big) \td t + \Big(\sqrt{\lambda(t)\big(p_2  + i(t,\hat{X}(t))^2 - 2p_1i(t,\hat{X}(t)\big)+ \varsigma^2}\Big) \td W(t).
\end{align*}
Under the assumption that the optimal control strategies exist for value function 
\begin{align}
    \label{EQ: value function without jumps}
\hat{V}(t,x) = \min_{l,i} E_{t,x} \Bigg[ & \tfrac{1}{2} \int_t^T e^{-\delta (s-t)} \left(l(s,\hat{X}(s)) - l_0(s)-l_1(s)\hat{X}(s)\right)^2 \td s  \nonumber \\
&\ + \tfrac{1}{2} \int_t^T e^{-\delta (s-t)} \gamma^i(s) i(s,\hat{X}(s))^2 \lambda(s)\td s \nonumber  \\
&\ + \tfrac{1}{2} \int_t^T e^{-\delta (s-t)} \left( \hat{X}(s) - x_0(s)\right)^2 \td \Gamma(s)   \nonumber      \\ 
&\ +\kappa e^{-\delta (T-t)} \left(\hat{X}(T)-x_T\right)^\tau \Bigg], 
\end{align}
for $t\leq T$ and $\hat{V}\in \mathcal{C}^{1,2}$, then the value function satisfies the HJB equation 
\begin{align*}
0 =&\  \hat{V}_t(t,x) - \delta \hat{V}(t,x) + \inf_{l,i} \Bigg\{ \  \frac{1}{2} \Big( l(t,x) - l_0(t) - l_1(t)x\Big)^2 \nonumber \\
    & \hspace{4 cm} + \frac{1}{2} \gamma^i(t) i^2 (t,x)\lambda(t) \nonumber \\
    & \hspace{4 cm} + \frac{1}{2} \gamma(t) \Big( x- x_0(t)\Big)^2  \nonumber \\
    & \hspace{4 cm} + \hat{V}_x(t,x) \Big( c(t) - l(t,x)\Big)\nonumber \\
     & \hspace{4 cm} + \frac{1}{2} \hat{V}_{xx}(t,x) \varsigma^2 \nonumber \\
     & \hspace{4 cm} +  \hat{V}_x(t,x)\lambda(t)(p_1-i(t,x)) \nonumber\\
     & \hspace{4 cm}+\frac{1}{2}\hat{V}_{xx}(t,x)\lambda(t)(p_2 + i(t,x)^2 - 2 p_1 i(t,x)) ,  \end{align*}
     with boundary condition
     \begin{align*}
         \hat{V}(T,x) =&\ \kappa (x - x_T)^\tau   + \Delta \Gamma(T)(x-x_0(T))^2. 
     \end{align*}

\end{corollary}
\begin{proof}
The surplus, $\hat{X}$, fits into the class of diffusion processes, where the drift and diffusion terms are deterministic functions of time, the surplus, and the control strategies. Furthermore, the value function is given by the expectation of Lebesgue integrals and functions of the terminal value of the surplus, since the Riemann-Stietjes integral only allow a final mass $\Delta \Gamma(T) \leq 0$. Therefore, the value function can be expressed in the form of the value function considered in \citet[Chapter 19]{Bj09} with the inclusion of a discount factor, and \citet[Chapter 19, Theorem 19.5 ]{Bj09} gives the result. 
\end{proof}
In Corollary \ref{coro: conside with surplus without jumps} the HJB equation is a necessary condition for the value function, but the HJB equation also take measures as a sufficient condition.
\begin{corollary}
\label{coro: verification conside with surplus without jumps}
Assume a function $\hat{H}$ satisfies the HJB equation
\begin{align*}
0 =&\  \hat{H}_t(t,x) - \delta \hat{H}(t,x) + \inf_{l,i} \Bigg\{ \  \frac{1}{2} \Big( l(t,x) - l_0(t) - l_1(t)x\Big)^2 \nonumber \\
    & \hspace{4 cm} + \frac{1}{2} \gamma^i(t) i^2 (t,x)\lambda(t) \nonumber \\
    & \hspace{4 cm} + \frac{1}{2} \gamma(t) \Big( x- x_0(t)\Big)^2  \nonumber \\
    & \hspace{4 cm} + \hat{H}_x(t,x) \Big( c(t) - l(t,x) + \lambda(t)(p_1-i(t,x))\Big)\nonumber \\
     & \hspace{4 cm} + \frac{1}{2} \hat{H}_{xx}(t,x)\Big( \varsigma^2 + \lambda(t)(p_2 + i(t,x)^2 - 2 p_1 i(t,x))  \Big),\nonumber \\   
    \end{align*}
    with boundary condition
    \begin{align*}
         \hat{H}(T,x) =&\  \kappa (x - x_T)^\tau   + \Delta \Gamma(T)(x-x_0(T))^2, 
    \end{align*}
and $\hat{H}_x(t,x)\varsigma \in \mathcal{L}^2$.
Furthermore, assume that the infimum is attained by admissible control strategies $\tilde{l}$ and $\tilde{i}$ for each fixed $(t,x)$. Then the optimal value function to the control problem, $\hat{V}$ from Equation \eqref{EQ: value function without jumps}, is
\begin{align*}
    \hat{V}(t,x)=\hat{H}(t,x),
\end{align*}
and the optimal control strategies are $l^*=\tilde{l}$ and $i^*=\tilde{i}$.
\end{corollary}
\begin{proof}
The surplus, $\hat{X}$, fits into the class of diffusion processes and the value function is given by the expectation of Lebesgue integrals and functions of the terminal value of the surplus. Hence, the value function can be expressed in the form of the value function considered in \citet[Chapter 19]{Bj09} with the inclusion of a discount factor, and \citet[Chapter 19, Theorem 19.6]{Bj09} gives the result. 
\end{proof}

The optimal dividend strategies that gives the infimum in the value function is given by Equations \eqref{Eq: optimal l affine} and \eqref{Eq: optimal i affine}. We see that the characterization of the value function and the optimal control in Corollary \ref{coro: verification conside with surplus without jumps} coincides with that of the original setup with jumps. Thus, for every model with jumps, there exists a pure diffusion model, constructed along the lines of Corollary \ref{coro: conside with surplus without jumps} and Corollary \ref{coro: verification conside with surplus without jumps}, such that the value function and the control are the same as for the model with jumps. Therefore, we consider from now on only pure diffusion models, and with reference to Corollary \ref{coro: conside with surplus without jumps} and Corollary \ref{coro: verification conside with surplus without jumps} we can say that this is without loss of generality towards inclusion of jumps in the model. We emphasize that this is due to the specific linear-quadratic form of the objective function and does not hold for a general objective function.

\section{Comparison to the classical objective} \label{Sec: comparison}
\subsection{The objectives}



In the introduction, we explained how \citet{deF57} advocated the use of the expected present value of dividend until ruin as a criterion of stability, and how this corresponded to the most basic model in finance for calculating the value of a project (NPV) or company \citep[Gordon model; see][]{Gor62}. This emphasizes focus on the expected present value of future dividends until ruin
\begin{align*}
    V^b(t,x) = \EE_{t,x}\Bigg[\int_t^{\tau_x} e^{-\tilde{\delta} (s-t)} \td D(s)\Bigg],
\end{align*}
where $\tau_x=\inf \{s\geq t : X(s) = 0 | X(t)=x\}$ is the time of ruin and $\tilde{\delta}$ is a financial impatience factor, not necessarily equal to $\delta$. 

In the classical set up, the company aims to distribute dividends such that the shareholders' dividend payouts, $V^b$, is maximized. The optimal dividend strategy that maximize $V^b$ is in general the band strategy, and reduces to the barrier strategy depending on the attainable values of $Y_i$ \citep[first studied by][in a discrete set up; references in continuous framework are provided in the introduction]{Mor66}. 
Realistic features for the dividend strategy are smooth and stable payouts that increase according to the surplus, and the dividend process is preferably non-decreasing. This encouraged the development of affine dividend strategies, which are shown in \cite{AvWo12} and \cite{AlCa17} to be close the barrier strategy, while improving stability significantly. Therefore, the affine dividend strategy is attractive even though, it is not the optimal dividend strategy in the classical objective.  In those papers, whether affine dividend strategies were mathematically optimal in some sense remained an open problem; we shed some light on this.

The value function in this paper is based on minimizing the dividend and surplus deviation from benchmarks by a quadratic loss function as described in Section \ref{SubSec: LQ objective}. This results in an optimal dividend strategy that is actually affine, such that we are able to obtain the affine dividend strategy as an optimal strategy, but the objective is not to maximize the present value of future dividends. The objective for the value function in the linear-quadratic optimization is to punish deviation from a benchmark, and thereby controlling the dividends and the surplus towards a target dividend payout and target surplus respectively. The explicit value of the value function does not assesses the companies financial situation as opposed to the expected present value of future dividends. We are interested in understanding the objective behind the value function and the resulting optimal dividend strategy. In other words the objective that is optimised focuses on what it incentivizes and  what the actual outcomes are, not the numerical value of the (quadratic) objective to minimize. 

There is no obvious way to compare the two approaches, as the objectives behind the value functions are different, making the value functions incomparable. One is not a special case of the other. Therefore, one goal of this paper is to understand the objective the company \textit{indirectly} optimises when it chooses to implement affine dividends as the optimal strategy. 
We examine how this relates to the classical objectives in various aspects. In order to compare the optimal dividend strategy under the LQ objective with that under the classical objective, and compare the present value of future dividends for each of the strategies, we need to determine suitable choices for the benchmarks.

In Section \ref{SubSec: benchmarks} we start by determining the expected present value of dividends in our context, but discuss why it is not reasonable to try and maximise it for comparison purposes, because of absence or ruin (and transaction costs). In the following section we explain how our benchmarks are determined by connecting our framework with the existing literature, which maximise the expected present value of dividends \textit{until ruin}.

\subsection{Benchmarks and the expected present value of future dividends}
\label{SubSec: benchmarks}
As argued earlier, the company may have a target dividend payout \citep[e.g.,][]{ShSt84,Kum88} and target surplus \citep[e.g.,][]{AI16}, and it would be natural for the company to distribute funds according to the optimal dividend strategy in the LQ problem with the targets as benchmarks. In what follows, we determine the expected present value of future dividends and how different choices of benchmarks affect this value.

In the following calculations, the surplus model and dividend payments are restricted to continuous payments, and we assume that the benchmark coefficients are constants. The optimal dividend strategy is then given by
\begin{align}
\label{Eq: optimal dividen strategy}
    l^*(t,x) = l_0 + p(t) + \big(l_1 + 2q(t)\big) x,
\end{align}
where $q(t)$ and $p(t)$ solves the differential equations from Proposition \ref{Prop: alt representation}, and where the surplus has dynamics
\begin{align}
\label{Eq: dynamic surplus with optimal dividend}
    \td X(t) = \big(c -l_0-p(t)\big)\td t - \big(l_1+2q(t)\big)X(t) \td t + \varsigma \td W(t).
\end{align}

We are interested in calculating the expected present value of future dividends, when the dividends are distributed according to the optimal dividend strategy in the LQ objective. This quantity,  
\begin{align}
\label{Eq: PV divd}
    V^{LQ}(t,x) = \EE_{t,x}\Bigg[\int_t^{T} e^{-\tilde{\delta} (s-t)} \bigg(l_0 + p(s) + \big(l_1 + 2q(s)\big) X(s) \bigg)\td s\Bigg]
\end{align}
up to a given time $T$, can be calculated via a PDE. The objective to maximize this function resembles the value function considered in \cite{AvWo12} and \cite{AlCa17} as discussed in Section \ref{SubSec: sub-optimal control problem}. \cite{AvWo12} and \cite{AlCa17} consider optimal parameters that maximize the expected present value of future dividends until ruin, where the dividend strategy is restricted to linear and affine in the surplus respectively. This is different from maximizing \eqref{Eq: PV divd} by the choice of benchmarks without stopping the process at ruin, even though the dividend strategy is affine.

\begin{proposition}
Assume $V^{LQ}(t, x)\in \mathcal{C}^{1,2}$. Then the expected present value of future dividends satisfies the following partial differential equation
\begin{align}
   V_t^{LQ}(t, x) =&\ \tilde{\delta} V^{LQ}(t, x) - l_0 - p(t) -\big(l_1 + 2q(s)\big) x - V_x(t,x) \bigg(c-l_0-p(t)-\big(l_1 + 2q(s)\big) x \bigg) \nonumber \\
   &\ -\frac{1}{2}V_{xx}(t,x) \varsigma^2, \nonumber \\
    V^{LQ}(T, x) =&\ 0.\label{Eq: PDE PV divd}
\end{align}

Conversely, if a function satisfies the partial differential equation above, it is indeed the expected present value of future dividends defined in Equation \eqref{Eq: PV divd}.
\label{Prop: Partial differential equation of PV divd}
\end{proposition}

\begin{proof}
See Appendix \ref{App: Proof of PV divd}.
\end{proof}

Similar to Section \ref{Subsec: alt representation}, we can express the expected present value of future dividends by a function that satisfies the PDE in Proposition \ref{Prop: Partial differential equation of PV divd}. We guess a solution 
\begin{align*}
    V^{LQ}(t, x) = f(t)x + g(t),
\end{align*}
where $f(t)$ and $g(t)$ satisfy the following differential equations
\begin{align*}
    f_t(t) =& f(t) \big( \tilde{\delta} + l_1 + 2q(t) \big) - l_1 - 2q(t) \\
    g_t(t) = & g(t) \tilde{\delta} + f(t) \big(l_0 + p(t) - c\big) - l_0 - p(t),
\end{align*}
with terminal conditions $f(T) = 0$ and $g(T)=0$. This function satisfies the PDE in Proposition \ref{Prop: Partial differential equation of PV divd} and is therefore the expected present value of future dividends. Hence, we can calculate the expected present value of future dividends for the optimal dividend strategy in the LQ problem by solving the ODEs for $f(t)$ and $g(t)$.

In order to determine suitable choices for the benchmarks, we consider how different alternatives affect the expected present value of future dividends. Figure \ref{Fig: Benchmark plot} illustrates the expected present value of future dividends, up to time $T=200$, for different benchmark values. We fix the values of the parameters that are not varied to Table \ref{Table: Parameters in example}  \citep[inspired by][]{AvWo12}. A low benchmark of the surplus causes the company to distribute a lot of dividends resulting in a high expected present value of future dividends. Similarly, larger benchmarks for the dividends in the form of either $l_0$ or $l_1$ cause the company to distribute more dividends. There is no constraint on the surplus being non-negative or the dividends being finite and non-negative. Therefore, based on Figure \ref{Fig: Benchmark plot}, the expected present value of future dividends is maximized by letting $x_0$ be zero and letting either $l_0$ or $l_1$ increase without bounds. This is unreasonable in practice, since the company is ruined if the surplus is negative.

\begin{figure}[htb]
    \centering
    \includegraphics[scale = 0.8]{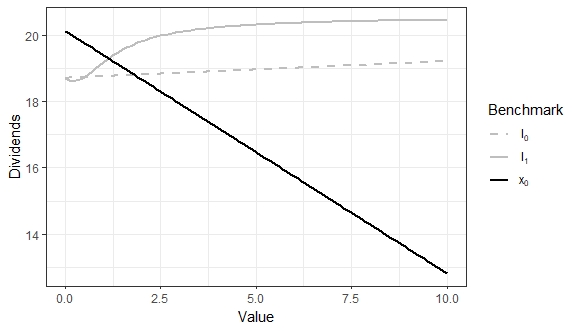}
    \caption{The expected present value of future dividends for different values of benchmarks. The grey and dashed line varies $l_0$, the grey and solid line varies $l_1$ and the black line is for varying values of $x_0$}
    \label{Fig: Benchmark plot} 
\end{figure}

\begin{table}[htb]
\begin{center}
\begin{tabular}{ l| l }
\hline
\textbf{Component} &  \textbf{Value} \\ \hline
$c(t)$ & 1 \\ 
$\varsigma$ & 0.5 \\ 
$\delta$  & 0.05 \\
$\tilde{\delta}$ & 0.05 \\
$T$ & 200 \\
X(0) & 0.628\\
    $l_0$ & 0 \\
 $l_1$ & 1/1.884\\ 
 $x_0$ & $1.884 $ \\
 $w^x(t)$ & 1 \\
 $\gamma$ & 0 \\
 $k$ & 0 \\
 $b^*$ & 1.256\\
\hline
\end{tabular}
\caption{Values of the components.}
\label{Table: Parameters in example}
\end{center}
\end{table}

The benchmarks and parameters in Table \ref{Table: Parameters in example} are the optimal linear strategy in \cite{AvWo12}. In order to get a better understanding of the resulting optimal affine dividend strategy in the LQ problem, where there is no constraint on non-negative dividends, and to compare it with linear strategy in \cite{AvWo12}, we illustrate the coefficient functions for the optimal dividend strategy $t \mapsto l_0 + p(t)$ and $t\mapsto l_1+2q(t)$, along with the benchmarks, for the parameters in Table \ref{Table: Parameters in example}. 

\begin{figure}[H]
    \centering
    \includegraphics[scale = 0.8]{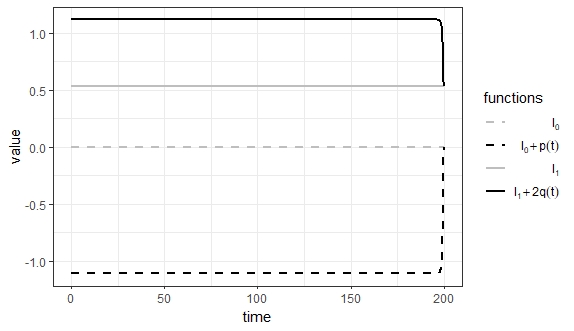}
    \caption{The coefficient functions for the optimal dividend strategy and the benchmarks. The grey lines are the benchmarks $l_0$ and $l_1$, the black lines are the coefficient functions. The solid lines are multiplied with the surplus in the dividend strategy and the dashed lines are added. }
    \label{Fig: coef plot} 
\end{figure}

Apart from termination, where the surplus benchmark is zero, the coefficients of the optimal control seem constant in Figure \ref{Fig: coef plot}. This is in correspondence with \citet{Cai00} for $T\rightarrow \infty$, because of the Markov property and the time-homogeneous objective. Based on Figure \ref{Fig: coef plot} a larger proportion of the surplus is distributed as dividends compared to the benchmark $l_1$.  It is adjusted by the part that is not multiplied with the surplus, $l_0 + p(t)$, which is negative. A negative value of dividends corresponds to a capital injection, and increases the surplus. Therefore, if the surplus is close to zero, the dividends are negative and a capital injection increases the surplus. The LQ objective is based on minimizing quadratic deviation, and therefore, it is indifferent if the dividends and surplus are above or below the benchmarks respectively. It would be reasonable to constrain the dividends to be non-negative in the objective or punish capital injection with a higher weight, but this would not lead to affine dividend strategies as studied in \citet{Stphd01}. It is a drawback in the LQ optimization problem that we do not avoid ruin or stop paying dividends after ruin, and there is no restriction on non-negative dividends. We do not incorporate suitable constraints to avoid the drawbacks, since the resulting optimal dividend strategy would not be affine.

\subsection{The sub-optimal control problem}
\label{SubSec: sub-optimal control problem}

In the classical objective the company aims to maximize the expected present value of future dividends until ruin, but for practical reasons the company is not interested in using the optimal dividend strategy, the barrier strategy \citep{Ger72}. The absence of constraint on non-negative surplus in the LQ objective makes it undesirable to use the benchmarks that maximize the expected present value of future dividends with the LQ optimal dividend strategy. \cite{AvWo12} and \cite{AlCa17} study the control problem of maximizing the expected present value of future dividends until ruin, where the dividend strategy is restricted to linear and affine in the surplus respectively 
\begin{align}
\label{Eq: mean reverting value function}
     \max_{\tilde{l}_0, \tilde{l}_1}\EE_{t,x}\Big[\int_t^{\tau_x} e^{-\tilde{\delta} (s-t)} \ \big(\tilde{l}_0 + \tilde{l}_1 X(s)\big)\ \td s\Big].
\end{align}
This results in optimal parameters, but the dividend strategy is not the optimal solution form of the classical optimal control problem, since the strategy class is specified. Both papers give an explicit expression for Equation \eqref{Eq: mean reverting value function} and numerically solve for the parameters that maximize this expression. By \cite{AlCa17} the optimal value of $\tilde{l}_0$ is zero in different numerical experiments \citep[reducing the problem to the framework of][]{AvWo12}. Hence, it is not desirable to pay immediate and fixed dividends when the surplus is close to zero. The objective is different from maximizing \eqref{Eq: PV divd}, since no more dividends are distributed if the company ruins, dividends are non-negative and the parameters in the affine dividends strategy are not maximized through quadratic differences to benchmarks.

The sub-optimal dividend strategies in \cite{AvWo12} and \cite{AlCa17} are close to the optimal barrier strategy and improve stability. We consider the optimal parameters from \cite{AvWo12} and \cite{AlCa17} as a suggestion for benchmarks, such that the benchmarks are the parameters that for an affine dividend strategy maximize the expected present value of future dividends until ruin. 

The surplus is an Ornstein-Uhlenbeck process when the dividend strategy is a continuous rate $\tilde{l}_1>0$ of the surplus
\begin{align}
    \td X(t) = c\ \td t - \tilde{l}_1X(t)\ \td t + \varsigma\ \td W(t) .
\end{align}
It reverts around the level $ c/\tilde{l}_1$, since the drift is positive, when the surplus is below the level, and negative above. The dividend strategy then reverts around $c$ and the strategy is therefore referred to as the mean reverting dividend strategy in \cite{AvWo12}. With the parameters of the optimization problem in Equation \eqref{Eq: mean reverting value function} as benchmarks in the LQ optimization problem, a suggestion for surplus benchmark is
\begin{align}
\label{Eq: surplus mean reversion benchmark}
    \tilde{x}_0= \frac{c}{\tilde{l}_1}.
\end{align}

The optimal dividend strategy for the LQ objective \eqref{Eq: optimal dividen strategy} implies that for $l_1+2q(t)>0$, the surplus has a positive drift at time $t$, when the surplus is below
\begin{align}
\label{Eq: surplus mean reversion}
    \frac{c -l_0-p(t)}{l_1+2q(t)},
\end{align}
and a negative drift, when the surplus is above Equation \eqref{Eq: surplus mean reversion}. Furthermore, the optimal dividend strategy increases and decreases accordingly to the surplus such that the dividend strategy reverts around $c$, the deterministic predictable modification component of the surplus. This is similar to the dividend strategy in \cite{AvWo12} and yields a smoother dividend flow around an annual rate $c$. In the next section we consider the optimal dividend strategy for the LQ objective with the optimal parameters from \cite{AvWo12} and \cite{AlCa17} as benchmarks in a numerical study.

\section{Numerical study}
\label{Sec: numerical}

\subsection{Present value of dividends by simulation}
\label{SubSec: simulation study}

Here we compare the barrier strategy and the mean reverting strategy to the dividend strategy in the LQ problem, where the mean reverting strategy is used as benchmark. 

We use the parameters in Table \ref{Table: Parameters in example} given in subsection \ref{SubSec: benchmarks}. The optimal barrier \citep[the one that maximises the expected present value of dividends until ruin, see][]{Ger72} for these values is $b^*=1.256$, and for $l_0=0$ the optimal level for the mean reverting strategy is $x_0= 1.884$ with $l_1=c/x_0$ \citep[the one that maximises the expected present value of dividends until ruin when a mean reverting strategy is forced, and pre-committed until ruin; see][]{AvWo12} and Equation \eqref{Eq: surplus mean reversion benchmark}. We simulate 2500 paths of the surplus for each of the three different dividend strategies using an Euler scheme with discretization step of $1/400$ up to time $T$. Furthermore, the surplus paths are stopped for the barrier and mean reverting strategy as soon as the surplus is negative or 0. All differential equations are solved using Runge-Kutta 4th order method, with the same discretization step as the simulated paths. 

Figure \ref{Fig: hexplot AFS b} displays a hexbin plot of 2500 outcomes of the present value of dividends for the barrier strategy and LQ strategy, and Figure \ref{Fig: hexplot AFS AW} displays a hexbin plot of 2500 outcomes of the present value of dividends for the mean reverting strategy and LQ strategy. The points above the 45 degrees line are the simulations, where the LQ optimal strategy outperforms the other strategy (respectively). The band around 10-20 on the vertical axis are the cases where the company is ruined before time $T$. In most cases the barrier strategy results in a higher present value of dividends, which is expected as it is the optimal dividend strategy. Figure \ref{Fig: hexplot AFS AW} shows that the LQ problem results in similar present value of future dividends as the mean reverting strategy, except for the cases where the company is ruined. Note the surplus for the strategy of the LQ objective is never negative, since the small discretization steps and negative values of $l_0+p(t)$ prevent this. Furthermore, there are only 6 paths out of the 2500 simulated paths, where no capital injection occurs with the LQ strategy.

\begin{figure}[H]
    \centering
    \includegraphics[scale = 0.8]{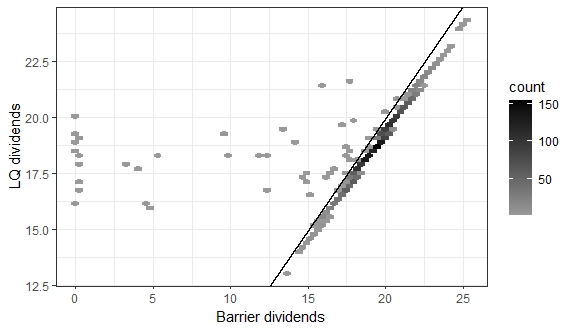}
    \caption{Hexbin plot of 2500 outcomes of the present value of dividends for the barrier strategy and LQ optimal strategy. The black line is the 45 degrees line.}
    \label{Fig: hexplot AFS b}
\end{figure}

\begin{figure}[H]
    \centering
    \includegraphics[scale = 0.8]{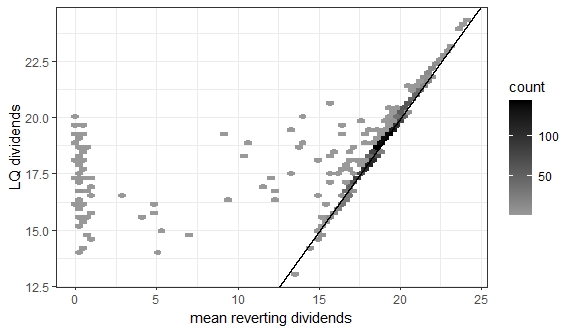}
    \caption{Hexbin plot of 2500 outcomes of the present value of dividends for the mean reverting strategy and LQ optimal strategy. The black line is the 45 degrees line.}
    \label{Fig: hexplot AFS AW}
\end{figure}

In most cases the LQ problem results in a present value of future dividends very close to the present value of future dividend, obtained by using the affine dividend strategy that maximize this quantity. This is illustrated in a violin plot in Figure \ref{Fig: violinplot} of the difference in the present value of dividends between the LQ strategy and mean reverting strategy. The violin plot is cut-off at -2, and do not show 83 cases out of the 137 cases, where the mean reverting strategy leads to ruin. Furthermore, the LQ problem minimize the chance of ruin, but this is done by capital injections, which is not desirable.

\begin{figure}[H]
    \centering
    \includegraphics[scale = 0.8]{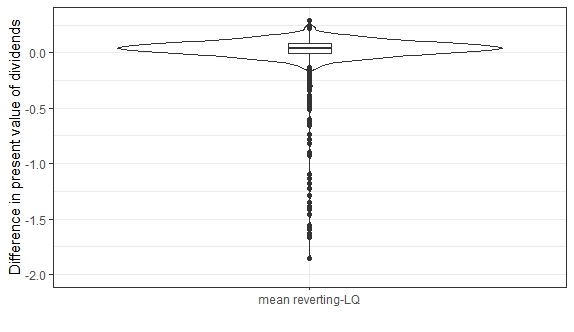}
    \caption{Violin plot that illustrate the distributions of 2500 simulated difference between present values of future dividends for the mean reverting strategy and the LQ strategy. }
    \label{Fig: violinplot}
\end{figure}

The Tables \ref{Table: diff x}, \ref{Table: diff sigma} and \ref{Table: diff delta} state mean and standard error of the 2500 simulations of present value of future dividends for each of the three dividend strategies for varying initial surplus, $\varsigma$ and $\tilde{\delta}$. We see that the present value of future dividends is in general higher for the barrier strategy, but the values for the mean reverting strategy and the LQ strategy are close. We also see that the standard error is smaller for the LQ strategy for all the different parameters, such that the deviation of the present value of future dividends is smaller for different simulation.

\begin{table}[H]
\begin{center}
\caption{Present value of future dividends varying initial surplus $x$, where $c=1, \tilde{\delta} =0.05, \sigma = 0.5$}
\label{Table: diff x}
\begin{tabular}{ l| l  l  l | l  l  l | l  l}
\hline
\textbf{x} & & mean &  & &  sd &  \\
& LQ & mean reverting & barrier  &  LQ & mean reverting & barrier & $b^*$ & $g^*$\\ \hline
0.1 b* & 18.054 & 11.441  & 12.483 & 1.511  & 8.879 & 9.139 & 1.256 &  0.465\\ 
0.5 b* & 18.727 & 18.279 & 19.179  & 1.522  & 3.231 & 2.381 & 1.256 & 0.531\\ 
1 b* &  19.502 & 19.323 & 20.018 & 1.512  & 2.096  & 1.859 & 1.256 & 0.560\\ 
1.5 b* & 20.077 & 19.900  & 20.614 & 1.556  & 1.99 & 1.816 & 1.256 & 0.570 \\ 
2 b* & 20.740 & 20.456 & 21.31  & 1.530 & 2.249 & 1.763  & 1.256 & 0.576\\ 
\hline
\end{tabular}
\end{center}
\end{table}

\begin{table}[H]
\begin{center}
\caption{Present value of future dividends varying $\varsigma$, , where $c=1, \tilde{\delta} =0.05, x = 0.5 b^* $}
\label{Table: diff sigma}
\begin{tabular}{ l| l  l  l | l  l  l | l  l}
\hline
\textbf{$\varsigma$}  & & mean &  & &  sd &\\
& LQ & mean reverting & barrier  &  LQ & mean reverting & barrier & $b^*$ & $g^*$\\ \hline
0.1  & 19.857 & 19.836  & 19.908 & 0.317  & 0.705 & 0.323 & 0.083 &  7.519\\ 
0.5  & 18.709 & 18.173 & 19.162  & 1.485  & 3.440 & 2.436 & 1.256 & 0.531\\ 
1  &  16.850 & 15.801 & 18.091 & 3.039  & 5.543  & 4.563 & 3.563 & 0.192\\
\hline
\end{tabular}
\end{center}
\end{table}

\begin{table}[H]
\begin{center}
\caption{Present value of future dividends varying $\tilde{\delta}$, where $c=1, \sigma =0.5, x = 0.5 b^* $ }
\label{Table: diff delta}
\begin{tabular}{ l| l  l  l | l  l  l | l  l}
\hline
\textbf{$\tilde{\delta}$}  & & mean &  & &  sd &\\
& LQ & mean reverting & barrier  &  LQ & mean reverting & barrier & $b^*$ & $g^*$\\ \hline
0.01  & 97.550 & 96.849  & 98.04 & 3.496 & 8.484 & 7.926 & 5.175 &  2.66\\ 
0.05  & 18.798 & 18.263 & 19.233  & 1.583  & 3.493 & 2.575 & 1.256 & 1.884\\ 
0.1  &  9.036 & 8.688 & 9.377 & 1.068  & 2.062  & 1.562 & 1.075 & 1.543\\
\hline
\end{tabular}
\end{center}
\end{table}

\subsection{The cost of smoothing dividends}

The expected present value of future dividends until ruin with the barrier strategy, $V^b$, is in general larger than the expected present value of future dividends with the LQ strategy, $V^{LQ}$, where the benchmarks are the parameters from the mean reverting strategy. The additional amount of initial surplus required for $V^{LQ}$ to be equal to $V^b$ is the cost of smoothing dividends as compared to the optimal barrier strategy. We calculate the extra amount of initial surplus,that we would need in order to achieve same expected value of future dividends. Hence we want to solve the following equation for $\xi$
\begin{align*}
    V^{LQ}(t, X(t)+\xi(t)) =V^b(t, X(t)).
\end{align*}
By the quadratic representation of $V^{LQ}$ we get that
\begin{align*}
    \xi(t) = \frac{V^b(t,X(t))-g(t)-f(t)X(t)}{f(t)},
\end{align*}
where 
\begin{align*}
    V^b(t,X(t))=\frac{e^{rX(t)}-e^{sX(t)}}{re^{rb^*}-se^{sb^*}}
\end{align*}
for the roots of $\frac{1}{2} \varsigma^2 z + c z - \tilde{\delta} = 0$
\begin{align*}
    r=& \frac{-c + \sqrt{c^2+2\tilde{\delta}\varsigma^2}}{\varsigma^2},\\
    s=& \frac{-c - \sqrt{c^2+2\tilde{\delta}\varsigma^2}}{\varsigma^2}.
\end{align*}
Figure \ref{Fig: initdiff} shows $\xi$ as a function of the initial surplus. The cost, $\xi$, is negative for very small values of $X(0)$, where the surplus with the barrier strategy is likely to ruin. It is not a fair comparison for small values of $X(0)$, as the LQ objective does not stop paying dividends after ruin and are likely to make capital injections. Apart from small values of $X(0)$, the additional amount of initial surplus needed for $V^{LQ}$ to be equal to $V^b$ increases as the initial surplus does. That is because the barrier strategy can pay an immediate dividend if the initial surplus is unnecessarily high. The LQ strategy can only pay proportionally. In those cases (high initial surplus) we expect this cost to be smaller for smaller $\delta$, larger risk, and/or larger payout targets.

\begin{figure}[H]
    \centering
    \includegraphics[scale = 0.7]{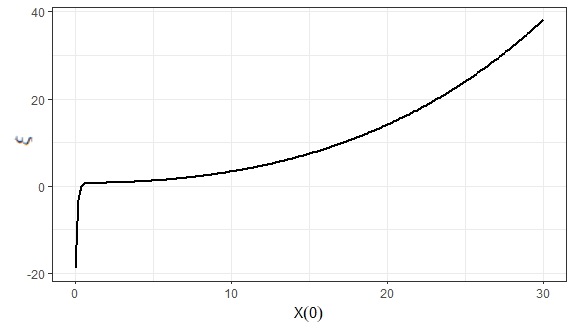}
    \caption{The extra amount of initial surplus, we would need in order to achieve same expected value of future dividends, considered as a function of the initial surplus.}
    \label{Fig: initdiff} 
\end{figure}




\section{Conclusion}

In this paper, we consider a very general surplus process that is modelled with two-sided jump Lévy dynamics. We control dividend payments in order to minimise quadratic differences to both payout and surplus level benchmarks, which can be altered according to a number of optional parameters.

We show that the dividend strategy that will minimise this LQ objective (quadratic differences) is of affine type (in the current surplus level). We futher characterise the optimal strategy parameters as the solutions of a well-defined system of ODEs. Remarkably, the optimal strategy is also shown to be the same as in the pure diffusion framework.

This is of particular interest, because there is evidence for companies to target payout ratios, and surplus levels, and the actual dividend strategies observed in practice are of affine type. Hence, we managed to create a framework, whereby objective and optimal control are \textit{both} realistic. This is generally not the case.

Finally, we calculate the expected present value of dividends when the optimal affine strategy is applied in a pure diffusion context. This is subsequently used to qualitatively and quantitatively compare results of the LQ optimisation with the classical dividend maximisation objective, thanks to simulations and some additional mathematical calculations.

\section*{Acknowledgments and declarations of interest }

Part of this work was done when Avanzi was visiting Steffensen the University of Copenhagen, and when Kusch Falden was visiting Avanzi at the University of Melbourne. The hospitality of the host institutions is gratefully acknowledged.

This research was  supported under the Australian Research Council Discovery (DP200101859) Project funding scheme. The views expressed herein are those of the authors and are not necessarily those of the supporting organisations. 

This research was funded by Innovation Fund Denmark award number 7076-00029. No potential conflict of interest was reported by the authors.

\section*{References}


\appendix

\section{Proof of Proposition \ref{prop: HJB equation}} \label{App: Proof of HJB equation}
The proof of the continuous parts is given in \citet[Chapter 19, Theorem 19.5]{Bj09}. Hence, it suffices to prove the result for the jumps, where
\begin{align*}
    c(t)=l(t,X(t))=l_0(t)=l_1(t)=\Gamma(t)=\varsigma=0,
\end{align*}
for all $t\geq 0$. Assume also initially that $\delta=0$ and let $(t,x)\in[0,T]\times \RR$ be fixed and define
\begin{align*}
&    \hat{i}(s,y) = \begin{cases}
i(s,y), \quad\quad (s,y) \in [t,t+h]\times \RR\\
i^*(s,y), \quad\quad  (s,y) \in (t+h,T]\times \RR
\end{cases}
\end{align*}
where $t+h<T$, $\bfi^*$ is the optimal strategy and $\bfi$ is an arbitrary fixed control strategy. Define
\begin{align*}
    \JJ(t,x, \hat{\bfi}) =& E_{t,x} \Bigg[  \frac{1}{2} \int_t^{T} \gamma^i(s) \hat{i}(s,X^{\hat{\textbf{i}}}(s-))^2 d N(s) + \kappa (X^{\hat{\textbf{i}}}(T)-x_T)^\tau \Bigg]\\
    = & \ 
    E_{t,x} \Bigg[  \frac{1}{2} \int_t^{t+h} \gamma^i(s) i(s,X^{\textbf{i}}(s-))^2 d N(s)\Bigg] + E_{t,x} \Bigg[ V(t+h,X^{\textbf{i}^*}(t+h))\Bigg] 
\end{align*}
where the last equation comes from partitioning the time interval and that the optimal stategy is used in the interval $(t+h,T]$. We assume $V$ is sufficiently regular and use Itôs lemma to obtain 
\begin{align*}
    V(t+h, X_{t+h}^{ \bfi}) =&\ V(t,x) +\int_{t}^{t+h}  V_t(s, X^{\bfi}(s))  \td s \\
    &\ + \int_{t}^{t+h}\bigg( V(s,X^{\bfi}(s))- V(s,X^{\bfi}(s-))  \big)\td N(s).
\end{align*}
The last term in $ \JJ(t,x, \hat{\bfi})$ is then given by
\begin{align*}
     E_{t,x}  \Bigg[  & V(t+h, X^{\textbf{i}}(t+h))\Bigg] =  V(t,x) + E_{t,x} \Bigg[ \int_{t}^{t+h}  V_t(s, X^{\bfi}(s)) \td s  \Bigg] \\
    &\ +  E_{t,x} \Bigg[ \int_t^{t+h} \lambda(s) \int_{-\infty}^{\infty} \Bigg(V(s,X^{\bfi}(s-)+y-i(s,X^{\bfi}(s-)) - V(s,X^{\bfi}(s-)) \Bigg) \td F_{Y}(y)\ \td s \Bigg]
\end{align*}
for $F_Y$ the distribution of any of the jump sizes, for instance $Y_1$, where we use the following is a martingale
\begin{align*}
    &\int_{t}^{t+h} \bigg(V(s,X^{\bfi}(s-) + \Delta X^{\bfi}(s))- V(s,X^{\bfi}(s-)) \bigg) \td N(s)\\
    &\ -  \int_t^{t+h} \lambda(s) \int_{-\infty}^{\infty} V(s,X^{\bfi}(s-)+y-i(s,X^{\bfi}(s-))- V(s,X^{\bfi}(s-))  \td F_{Y}(y)\ \td s.
\end{align*}

Note the optimal control strategy minimize $\JJ$ therefore
\begin{align*}
    V(t,x)=\JJ(t,x, \bfi^*) \leq  \JJ(t,x, \hat{\bfi}),
\end{align*}
with equality if and only if the control strategy is the optimal $\bfi=\bfi^*$. The inequality is

\begin{align*}
    V(t,x) \leq & E_{t,x} \Bigg[  \frac{1}{2} \int_t^{t+h} \gamma^i(s) i(s,X^{\textbf{i}}(s-))^2 d N(s)\Bigg]+ V(t,x)+ E_{t,x} \Bigg[ \int_{t}^{t+h}  V_t(s, X^{\bfi^*}(s)) \td s  \Bigg] \\
    &\ +  E_{t,x} \Bigg[ \int_t^{t+h}\lambda(s) \int_{-\infty}^{\infty} \Bigg(V(s,X^{\bfi}(s-)+y-i^*(s,X^{\bfi^*}(s-)) - V(s,X^{\bfi^*}(s-)) \Bigg) \td F_{Y}(y)\ \td s \Bigg]
\end{align*}
 
 Dividing by $h$ and assume sufficient regularity to consider the limit $h\rightarrow 0$ within the expectation
\begin{align*}
    0\leq &  V_t(t, x) 
+ \frac{1}{2} \lambda(t)  \gamma^i(t) i^2(t,x)  +\lambda(t) E\big[V(t, x + Y_1 - i(t,x) ) - V(t, x)\big]
\end{align*}
where we use that $X^{\bfi}(t-)=x$. Since the control strategy $\bfi$ is arbitrary, this inequality holds for all choices of control strategies, therefore also the strategy that gives the infimum of this quantity
\begin{align*}
    0\leq & \inf_{i} \bigg\{ V_t(t, x) 
+ \frac{1}{2} \lambda(t)  \gamma^i(t) i^2(t,x)  +\lambda(t) E\big[V(t, x + Y_1 - i^*(t,x) ) - V(t, x)\big] \bigg\}.
\end{align*}
We also have that infimum over an arbitrary strategy, must be smaller than any other strategy also the optimal strategy, hence, $\inf_i \JJ(t,x, \hat{\bfi}) \leq \JJ(t,x, \bfi^*)$, which gives the other inequality, and HJB equation with $\delta=0$. 

For $\delta\neq 0$ 
\begin{align*}
    V(t,x) =& \min_i E_{t,x} \Bigg[  \frac{1}{2} \int_t^{T}  e^{-\delta (s-t)}\gamma^i(s) i(s,X(s-))^2 d N(s) + e^{-\delta(T -t)}\kappa (X^{\bfi}(T)-x_T)^\tau \Bigg]\\
    =& e^{\delta( t)} \tilde{V}(t,x).
\end{align*}
By the previous results, the function $\tilde{V}(t,x)$ satisfy
\begin{align*}
    0 =  \tilde{V}_t(t, x) + \inf_i \bigg\{ e^{- \delta s}\frac{1}{2} \lambda(t)  \gamma^i(t) i^2(t,x)  +\lambda(t) E\big[\tilde{V}(t, x + Y_1 - i(t,x) ) - \tilde{V}(t, x)\big]\bigg\},
\end{align*}
where $\tilde{V}_t(t, x)= e^{-\delta t}V_t(t,x) - \delta e^{-\delta t}V(t,x) $. Multiply everthing by $e^{\delta t}$ to obtain the desired.

\section{Proof of Proposition \ref{prop: Verification}} \label{App: Proof of verification}
The proof of the continuous parts is given in \citet[Chapter 19, Theorem 19.6]{Bj09}. Hence, it suffices to prove the result for the jumps, where
\begin{align*}
    c(t)=l(t,X(t))=l_0(t)=l_1(t)=\Gamma(t)=\varsigma=0.
\end{align*}

Assume $H\in C^1$, $C^2$ except at countable many points and solves the HJB equation. Furthermore, assume for fixed $(t,x)\in[0,T]\times \RR$ that the control strategy $\tilde{i}(t,x)$ minimizes the HJB equation of $H$. 

Let $\bfi$ be an arbitrary control strategy and $X^{\textbf{i}}(s)$ be the surplus with dynamics
\begin{align*}
    \td X^{\textbf{i}}(s) =   \td S(s) - i(s,X^{\textbf{i}}(s-))\td N(s),
\end{align*}
for $t\leq s \leq T$ and $X^{\textbf{i}}(t-)=x$.

Since $H$ solves the HJB equation we have for $t\leq s \leq T$ 
\begin{align*}
0 \leq & H_t(s,X^{\textbf{i}}(s-)) - \delta H(s,X^{\textbf{i}}(s-)) \nonumber \\
    & \ + \frac{1}{2} \gamma^i(s) i^2(s,X^{\textbf{i}}(s-)) \lambda(s) \nonumber \\
     & \ +  \lambda(s) \EE \Big[H(s, X^{\textbf{i}}(s-)+ Y_{1}-i(s,X^{\textbf{i}}(s-))) - H(s, X^{\textbf{i}}(s-)) \Big], 
    \end{align*} 
  $P-a.s.$ for all possible control strategies $\bfi$. We consider the integral over $(t,T]$ for both sides of the inequality multiplied by the positive function $e^{-\delta(s-t)}$ for $t\leq s\leq T$. 
    \begin{align*}
0 \leq & \int_t^T e^{-\delta(s-t)}H_t(s,X^{\textbf{i}}(s-)) \td s -\delta \int_t^T e^{-\delta(s-t)}H(s,X^{\textbf{i}}(s-)) \td s  \\
    & \ + \frac{1}{2} \int_t^T e^{-\delta(s-t)}\gamma^i(s) i^2(s,X^{\textbf{i}}(s-)) \Big(\lambda(s) \td s - \td N(s)\Big)  + \frac{1}{2} \int_t^T e^{-\delta(s-t)} \gamma^i(s) i^2(s,X^{\textbf{i}}(s-))  \td N(s) \nonumber \\
     & \ +  \int_t^T e^{-\delta(s-t)} \lambda(s) \EE \Big[H\big(s, X^{\textbf{i}}(s-)+ Y_{1}-i(s,X^{\textbf{i}}(s-))\big) - H(s, X^{\textbf{i}}(s-)) \Big] \td s. 
    \end{align*}

Itôs lemma implies that
\begin{align*}
    e^{-\delta(T-t)} H(T, X_{T}^{ \bfi}) =&\ H(t,x)  - \delta \int_t^T e^{-\delta(s-t)}H(s,X^{\textbf{i}}(s)) \td s \\
    &\     +\int_{t}^{T}  e^{-\delta(s-t)}  H_t(s, X^{\bfi}(s)) \td s \\
    &\ + \int_{t}^{T}  e^{-\delta(s-t)} \bigg( H(s,X^{\bfi}(s))- H(s,X^{\bfi}(s-)) \bigg) \td N(s), 
\end{align*}
such that
 \begin{align*}
H(t,x) \leq &   e^{-\delta(T-t)}H(T,X^{\textbf{i}}(T))    \nonumber \\
    & \ + \frac{1}{2} \int_t^T  e^{-\delta(s-t)} \gamma^i(s) i^2(s,X^{\textbf{i}}(s-)) \Big(\lambda(s) \td s - \td N(s)\Big)\\
    &\ + \frac{1}{2} \int_t^T  e^{-\delta(s-t)}\gamma^i(s) i^2(s,X^{\textbf{i}}(s-))  \td N(s) \nonumber \\
     & \ +  \int_t^T  e^{-\delta(s-t)}\lambda(s) \EE \Big[H\big(s, X^{\textbf{i}}(s-)+ Y_{1}-i(s,X^{\textbf{i}}(s-))\big) - H(s, X^{\textbf{i}}(s-)) \Big] \td s
     \\
     & \ -  \int_{t}^{T} e^{-\delta(s-t)} \bigg( H(s,X^{\bfi}(s))- H(s,X^{\bfi}(s-)) \bigg) \td N(s) . 
    \end{align*} 

Take $\EE_{t,x}[\cdot]$ on both sides of the inequality. 

 \begin{align*}
H(t,x) \leq &  \EE_{t,x}\Bigg[ e^{-\delta(T-t)}\kappa \bigg(X^{\textbf{i}}(T)-k\bigg)^\tau\\
    & \ + \frac{1}{2} \int_t^T  e^{-\delta(s-t)}\gamma^i(s) i^2(s,X^{\textbf{i}}(s-))  \td N(s) \Bigg] \nonumber \\
    &\ +  E_{t,x} \Bigg[ \int_t^Te^{-\delta(s-t)} \lambda(s) \int_{-\infty}^{\infty} H\big(s,X^{\bfi}(s-)+y-i(s,X^{\bfi}(s-)\big)- H(s,X^{\bfi}(s-))  \td F_{Y}(y)  
     \\
     & \ -  \int_{t}^{T}e^{-\delta(s-t)} \bigg( H(s,X^{\bfi}(s))- H(s,X^{\bfi}(s-)) \bigg) \td N(s) \Bigg]\\
     &\ = \JJ(t,x,\bfi), 
    \end{align*} 
Where we use the compensated jump measure $\EE\left[\td N(s)\right] = \lambda(s)\td s$ and the compensated martingale 
\begin{align*}
    &\int_{t}^{t+h} H\big(s,X^{\bfi}(s-) + \Delta X^{\bfi}(s)\big)- H(s,X^{\bfi}(s-))  \td N(s)\\
    &\ -  \int_t^{t+h} \lambda(s) \int_{-\infty}^{\infty} H\big(s,X^{\bfi}(s-)+y-i(s,X^{\bfi}(s-)\big)- H(s,X^{\bfi}(s-))  \td F_{Y}(y)\ \td s.
\end{align*}

Since the control strategy is arbitrary we also have that
    \begin{align}
    \label{eq: first inequality}
        H(t,x) \leq \inf_{l,i}\ \JJ(t,x,\bfi) = V(t,x).
    \end{align}

For the optimal control strategy $\tilde{\bfi}$ the HJB equation implies $P$-a.s.
\begin{align*}
0 = & H_t(s,X^{\tilde{\bfi}}(s-)) - \delta H(s,X^{\tilde{\bfi}}(s-)) \nonumber \\
    & \ + \frac{1}{2} \gamma^i(s) \tilde{i}^2(s,X^{\tilde{\bfi}}(s-)) \lambda(s) \nonumber \\
     & \ +  \lambda(s) \EE \Big[H\big(s, X^{\tilde{\bfi}}(s-)+ Y_{1}-\tilde{i}(s,X^{\tilde{\bfi}}(s-))\big) - H(s, X^{\tilde{\bfi}}(s-)) \Big], 
    \end{align*} 

Consider the integral over $(t,T]$ for both sides of the equality multiplied by the positive function $e^{-\delta(s-t)}$ for $t\leq s\leq T$, and the expression of Itôs lemma
    \begin{align*}
H(t,X^{\tilde{\bfi}}(t)) = &  e^{-\delta(T-t)} H(T,X^{\tilde{\bfi}})  \\
    & \ + \frac{1}{2} \int_t^T e^{-\delta(s-t)} \gamma^i(s) \tilde{i}^2(s,X^{\tilde{\bfi}}(s-))\Big(\lambda(s) \td s - \td N(s)\Big)\\
    &\ + \frac{1}{2} \int_t^T  e^{-\delta(s-t)}\gamma^i(s) \tilde{i}^2(s,X^{\tilde{\bfi}}(s-))  \td N(s) \nonumber \\
     & \ +  \int_t^T e^{-\delta(s-t)}\lambda(s) \EE \Big[H\big(s, X^{\tilde{\bfi}}(s-)+ Y_{1}-\tilde{i}(s,X^{\tilde{\bfi}}(s-))\big) - H(s, X^{\tilde{\bfi}}(s-)) \Big] \td s
     \\
     & \ -  \int_{t}^{T}e^{-\delta(s-t)} \bigg( H(s,X^{\tilde{\bfi}}(s))- H(s,X^{\tilde{\bfi}}(s-)) \bigg) \td N(s) . 
    \end{align*}

Take $\EE_{t,x}[\cdot]$ on both sides of the equality
   \begin{align*}
H(t,x) = &  \EE_{t,x}\Bigg[ e^{-\delta(T-t)}\kappa \bigg(X^{\textbf{i}}(T)-k\bigg)^\tau  + \frac{1}{2} \int_t^T  e^{-\delta(s-t)}\gamma^i(s) \tilde{i}^2(s,X^{\tilde{\bfi}}(s-))  \td N(s) \nonumber \Bigg]\\
      &\ +  E_{t,x} \Bigg[ \int_t^Te^{-\delta(s-t)} \lambda(s) \int_{-\infty}^{\infty} H\big(s,X^{\tilde{\bfi}}(s-)+y-i(s,X^{\tilde{\bfi}}(s-)\big)- H(s,X^{\tilde{\bfi}}(s-))  \td F_{Y}(y)  
     \\
     & \ -  \int_{t}^{T}e^{-\delta(s-t)} \bigg( H(s,X^{\tilde{\bfi}}(s))- H(s,X^{\tilde{\bfi}}(s-)) \bigg)  \td N(s) \Bigg]\\
     &\ = \JJ(t,x,\tilde{\bfi}).
    \end{align*} 
We must have that
\begin{align*}
    H(t,x) = \JJ(t,x,\tilde{\bfi}) \geq \inf_{l,i} \JJ(t,x,\bfi) = V(t,x),
\end{align*}
which together with \eqref{eq: first inequality}  shows that
\begin{align*}
    H(t,x) = \JJ(t,x,\tilde{\bfi}) = V(t,x),
\end{align*}
and $\tilde{l}$ and $\tilde{i}$ are the optimal control strategies.

\section{ODEs for $q$, $p$ and $r$} \label{App: q p r}

\begin{align}
    q_t(t) =&\delta  q(t)+ 2q(t)^2- \frac{1}{2} \gamma^{x}(t) - \Big( \gamma^i(t)+2q(t) \Big) 2\lambda(t)  \frac{q(t)^2}{\big(\gamma^i(t)+2q(t)\big)^2} + 2q(t)l_1(t) \nonumber \\
   &\  +4\lambda(t)\frac{q(t)}{\gamma^i(t)+2q(t)}\label{Eq: q}\\
   p_t(t) =& \delta p(t)+\gamma(t)x_0(t) - \Big(\gamma^i(t)+q(t)\big)\lambda(t)\frac{2q(t)p_1+p(t)}{\gamma^i(t)+2q(t)} \frac{2q(t)}{\gamma^i(t)+2q(t)} \nonumber \\
   &\ -2q(t)c(t) -2q(t)\lambda(t)p_1 + 2q(t)l_0(t) + p(t)l_1(t) + 2q(t)p(t) \nonumber\\
   &\  - (1-2p_1 + p(t))\lambda(t) \frac{2q(t)}{\gamma^i(t)+2q(t)} + 2\lambda(t)\frac{2q(t)p_1+p(t)}{\gamma^i(t)+2q(t)} \label{Eq: p} \\
   r_t(t)=&\delta r(t)+ \frac{1}{2} p(t)^2 -\frac{1}{2}\gamma(t)x_0(t)^2 -\Big(\frac{1}{2}\gamma^i(t) + q(t)\Big) \lambda(t) \Big(\frac{1-2q(t)p_1-p(t)}{\gamma^i(t)+2q(t)}\Big)^2 \nonumber\\
   &\ -p(t)c(t) -p(t)\lambda(t)p_1 +p(t) l_0(t)  -2q(t)\varsigma^2 + \Big(2p_1-1 - p(t)\Big)\lambda(t) \frac{2q(t)p_1+p(t)}{\gamma^i(t)+2q(t)}. \label{Eq: r}
\end{align}
with terminal conditions
\begin{table}[H]
\begin{center}
\caption{Terminal conditions }
\label{Table: terminal conditions}
\begin{tabular}{ l| l | l | l }
\hline
 $\tau$ & 0 & 1 &  2 \\ \hline
$q(T)$ & $\Delta \Gamma(T)$ & $\Delta \Gamma(T)$  & $\Delta \Gamma(T)+\kappa$ \\ 
$p(T)$ & $-2\Delta \Gamma(T)x_0(T)$ & $-2\Delta \Gamma(T)x_0(T)+\kappa$ & $-2\Delta \Gamma(T)x_0(T)-2\kappa x_T$ \\ 
$r(T)$ &  $\Delta \Gamma(T)x_0(T)^2$ & $\Delta \Gamma(T)x_0(T)^2-\kappa x_T$ & $\Delta \Gamma(T)x_0(T)^2+\kappa x_T^2$\\ 
\hline
\end{tabular}
\end{center}
\end{table}

Remember that $\tau$ is a parameter of the terminal value constraint in \eqref{Eq: value function}; see Section \ref{SubSec: LQ objective}.

\section{Proof of Proposition \ref{Prop: Partial differential equation of PV divd}} \label{App: Proof of PV divd}

Construct a martingale $m$ as
\begin{align*}
    m(t) =&\ \EE^{\PQ} \Bigg[ \int_0^T e^{- \tilde{\delta} s} \bigg(l_0 + p(s) + \big(l_1 + 2q(s)\big) X(s) \bigg)\td s \ \Bigg| \ \mathcal{F}_t \ \Bigg] \\
    =&\ \int_0^t e^{- \tilde{\delta} s} \bigg(l_0 + p(s) + \big(l_1 + 2q(s)\big) X(s) \bigg)\td s + e^{- \tilde{\delta} t} V^{LQ}(t, X(t)).
\end{align*}

The dynamics of $m$ are
\begin{align*}
    \td m(t) =&\ e^{- \tilde{\delta} t} \Big( l_0 + p(t) + \big(l_1 + 2q(t)\big) X(t)  -  \tilde{\delta} V^{LQ}(t, X(t)) \td t \\
    &\ \hspace{2cm} + \td V^{LQ}(t, X(t)) \Big). 
\end{align*}

By the Itô formula, we have the dynamics  
\begin{align}
    \td V^{LQ}(t&, X(t)) \nonumber \\
    =&\ V^{LQ}_t(t, X(t)) \td t + V_x^{LQ}(t, X(t)) \Big(c- l_0 - p(t) - \big(l_1 + 2q(t)\big) X(t)\Big) \td t \nonumber \\
    &\  +\frac{1}{2} V_{xx}^{LQ}(t, X(t)) \varsigma^2 \td t + V_x^{LQ}(t, X(t)) \varsigma \td W(t)  
   \label{Eq: Dynamics V in proof of PDE} 
\end{align}

Combining this, the dynamics of $m(t)$ are
\begin{align*}
    \td m(t) =&\ e^{- \tilde{\delta} t} \Bigg( l_0 + p(t) + \big(l_1 + 2q(t)\big) X(t)  -  \tilde{\delta} V^{LQ}(t, X(t))  \\
    &\ \hspace{1cm} + \ V^{LQ}_t(t, X(t)) + V_x^{LQ}(t, X(t)) \Big(c- l_0 - p(t) - \big(l_1 + 2q(t)\big) X(t)\Big)  \nonumber \\
    &\   \hspace{1cm} +\frac{1}{2} V_{xx}^{LQ}(t, X(t)) \varsigma^2 \Bigg)\td t\\
    &\ + e^{- \tilde{\delta} t}V_x^{LQ}(t, X(t)) \varsigma \td W(t). 
\end{align*}

Since $e^{- \tilde{\delta} t}V_x^{LQ}(t, X(t)) \varsigma \td W(t)$ are the dynamics of a martingale and since $m(t)$ is a martingale, the term in front of $\td t$ in the dynamics of $m(t)$ must be equal to zero for all $t$ and $X(t)$ which results in the partial differential equation for the expected present value of future dividends. By the expression og $V^{LQ}$ the boundary condition of the partial differential equation is $V^{LQ}(T, x) = 0$.

Now, assume that a function $\bar{V}^{LQ}(t, x)$ satisfies the partial differential equation in Equation \eqref{Eq: PDE PV divd}. We show that this function is in fact the expected present value of future dividends in Equation \eqref{Eq: PV divd}. 

The Itô formula, the dynamics from Equation \eqref{Eq: Dynamics V in proof of PDE} with $\bar{V}$ inserted instead of $V$, and the fact that $\bar{V}$ satisfies the partial differential equation in Equation \eqref{Eq: PDE PV divd} yield that
\begin{align*}
    &\ \td \Big( e^{- \tilde{\delta} t} \bar{V}^{LQ}(t, X(t)) \big) \\ 
    =&\ -\tilde{\delta} \bar{V}^{LQ}(t, X(t)) \td t + e^{- \tilde{\delta} t} \td \bar{V}^{LQ}(t, X(t)) \\
    =&\ e^{- \tilde{\delta} t} \Bigg( \Big( l_0 + p(t) + \big(l_1 + 2q(t)\Big) X(t) \big) \td t   + V_x^{LQ}(t, X(t)) \varsigma \td W(t)  \Bigg).
\end{align*}

Integrating over the interval $[t, T)$ and taking the expectation conditioning on $\mathcal{F}_t$ give that 
\begin{align*}
    &\ e^{- \tilde{\delta} T}\bar{V}^{LQ}(T, X(T)) - e^{- \tilde{\delta} t} \bar{V}^{LQ}(t, X(t)) \\
    =&\ - \EE\bigg[ \int_t^T e^{- \tilde{\delta} s} \Big( l_0 + p(s) + \big(l_1 + 2q(s)\big) X(s) \Big) \td s \ \bigg| \ \mathcal{F}_t \ \bigg],
\end{align*}

since the remaining term in the dynamics of $\bar{V}^{LQ}(t, X(t))$ is a martingale with respect to the filtration $\mathcal{F}$. Multiplying by $- \exp(- \tilde{\delta} t)$ gives that $\bar{V}^{LQ}(t, x)$ is the expected present value of future dividends.

\end{document}